\documentclass[12pt]{article}
\usepackage{a4wide}
\usepackage{euscript}
\usepackage{fullpage}
\usepackage[T2A]{fontenc}
\usepackage[utf8]{inputenc}
\usepackage[english]{babel}
\usepackage{amsfonts}
\usepackage{amssymb, amsthm}
\usepackage{amsmath}
\usepackage{mathtools}
\usepackage{graphicx}
\usepackage{geometry}
\usepackage{tikz}
\usepackage{bbm}

\newtheorem{st}{Statement}
\newcommand{\Z}{\mathbb{Z}}
\newcommand{\R}{\mathbb{R}}

\newcommand{\bx}{\mathbf{x}}

\newcommand{\Prob}{\mathbb{P}}
\newcommand{\D}{\mathbb D}

\newtheorem{lm}{Lemma}
\newtheorem{tm}{Theorem}
\newtheorem{cor}{Corollary}

\newcommand{\area}{\mathrm{area}\,}
\newcommand{\per}{\mathrm{per}\,}
\newcommand{\vol}{\mathrm{vol}\,}
\newcommand{\cov}{\mathrm{cov}\,}

\newcommand{\deq}{\stackrel{D}{=}}
\newcommand{\ceq}{\stackrel{C}{=}}
\newcommand{\cdeq}{\stackrel{CD}{=}}
\newcommand{\E}{\mathbb E}
\newcommand{\ind}{\mathbbm 1}

\begin{document}

\pagestyle{plain}
\title{Lattice points inside a random shifted integer polygon}
\date{}
\author{Aleksandr Tokmachev\thanks{The work was supported by the Theoretical Physics and Mathematics Advancement Foundation «BASIS».}}
\maketitle
\begin{abstract}
Consider a convex body $C \subset \R^d$. Let $X$ be a random point with uniform distribution in $[0,1]^d$.  Consider the value $X_C$ equal to the number of lattice points $\Z^d$ inside the body $C$ shifted by $X$. It is well known that $\E X_C = \vol(C)$. The question arises: what can be said about the variance of this random variable? This paper answers this question in the case when $C$ is a polygon with vertices at integer points. Moreover, an explicit distribution of $X_T$ is given for the integer triangle $T$.
\end{abstract}

\section{Introduction}
Consider the $d$-dimensional space $\R^d$ and the lattice $\Z^d$. Let $C$ be a convex body in $\R^d$, i.e., a convex compact with nonempty interior. Consider a random point $X$ with uniform distribution inside the unit cube $[0,1]^d$. Let $X_C$ be a number of lattice points inside $C + X$:
\begin{align}
    X_C = |(C + X) \cap \Z^d|,  \nonumber
\end{align}
where $X$ is a random point uniformly distributed in $[0,1]^d$. It is a well known result that the mathematical expectation of $X_C$ is equal to the volume of the $C$:
\begin{align}
    \E X_C = \vol (C).
\end{align}

The question arises: what can we say about the variance of $X_C$? In the case when $C$ is a disk of large radius this question is related with the \textit{circle problem}, which asks for the number of lattice points in the disk $rB^2$, of radius $r$ centred at the origin. The author directs the reader to \cite{Strombergsson} for further insight into this connection and potential generalizations. Additionally, we cite the paper \cite{Janacek}, which examines the asymptotic of the variance of $rC$ in the case when the body $C$ differs from the ball.

In this paper, we consider the two--dimensional case where the body is an integer polygon, that is, a polygon whose vertices all have integer coordinates. To underscore the fact that we are discussing a two-dimensional polygon, we will denote the body as $P$. In this case, the mathematical expectation of $X_P$ is equal to the area of $P$, which can be calculated using the number of points inside the polygon and on its sides according to Pick's theorem [3]. We will prove that the variance depends only on the number of points on the sides of the polygon. Moreover, we obtain an explicit formula for the distribution of $X_P$ when $P$ is an integer triangle.

The paper is organized as follows: Section 2 provides the necessary definitions and basic statements about the distribution of $X_P$. Section 3 discusses results related to the variance of $X_P$. Section 4 presents the explicit distribution for a triangle with integer vertices. Finally, the remaining sections give complete proofs of the formulated lemmas and theorems.

\section{Preliminaries}
We will begin by establishing the notations that will be employed throughout this paper.

\textbf{Notations.}
\begin{itemize}
\item An \textit{integer point} is defined as a point whose both coordinates are integers. Similarly, an \textit{integer vector or polygon} is defined as a vector or polygon with integer vertices. We identify integer vectors and integer points as follows: $m \leftrightarrow \overrightarrow{0m}$.

\item We use the standard notation $X \deq Y$ for identically distributed random variables and use $X \ceq Y$ for random variables such that $X - \E X = Y - \E Y $ a.s. Note that $X \ceq Y$ iff exist a constant $C$ such that $X = Y + C$ a.s. We also write $X \cdeq Y$ if $X - \E X \deq Y - \E Y$. 

\item For two sets $A$ and $B$ we denote by $A \oplus B$ the Minkowski sum of these sets:
$$  
    A \oplus B = \{a + b \colon a \in A,~b \in B\}.
$$

\item For two vectors \( v \) and \( w \), let \( v \oplus w \) denote the parallelogram whose vertices are \( 0 \), \( v \), \( w \), and \( v + w \), considered as points, and let \( v_1 \wedge v_2 \) denote the oriented area of \( v_1 \oplus v_2 \): if \( v_1 = (x_1, y_1) \) and \( v_2 = (x_2, y_2) \), then \( v_1 \wedge v_2 = x_1y_2 - x_2y_1 \).

\item In the context of a polygon, the term "sides" is used to refer to a set of vectors representing the counterclockwise--oriented sides of the polygon. We call the \textit{affine length} of a side of an integer polygon the number of segments into which integer points divide that side. We call an \textit{affine perimeter} the sum of the affine lengths of sides.
\end{itemize}

Let $v \in \Z^2$ be an integer vector. Denote by $X_v(x)$ the number of integer points inside the parallelogram $x \oplus v$ with ''orientation'':
$$
    X_v(\bx) =
    \begin{cases}
        |(\bx \oplus v) \cap \Z^2| + |v|_{aff} \quad &\text{if } \bx \wedge v \geqslant 0,\\
        -|(\bx \oplus v) \cap \Z^2| \quad &\text{if } \bx \wedge v < 0.
    \end{cases}
$$  
It is handy to use another form for the distribution of $X_v$. 
\begin{st}\label{vecdist}
    For any integer vector $v \in \Z^2$ the following holds 
    \begin{align}
        X_v \ceq |v|_{aff}\left\lceil X \wedge \frac{v}{|v|_{aff}}  \right\rceil,
    \end{align} 
    where \(\lceil \cdot \rceil\) denotes the ceiling function.
\end{st}

The following statement contains the simplest properties of $X_P$.

\begin{st}\label{baza} Let an integer polygon $P$ with sides $v_1, v_2, \ldots v_n$ be given. Then:
    \begin{enumerate}
    
        \item $\E X_P = \area (P)$.

        \item For any $A \in SL(2, \Z)$, $$X_{AP} \deq X_P.$$

        \item $X_{-P} \deq X_P$.

        \item $X_P \ceq  X_{v_1} + X_{v_2} + \ldots + X_{v_n}.$

        \item For any integer polygons $P$ and $Q$,
        $$
            X_{P \oplus Q} \ceq X_P + X_Q, 
        $$
        where $P \oplus Q$ is a Minkowski sum of $P$ and $Q$.

         \item If $P$ is centrally symmetric, then $X_P = const$ a.s.

        \item $X_P \cdeq -X_p$, i.e. the random variable $X_P - \E X_P$ has a symmetric distribution. 
        
    \end{enumerate}
\end{st}
The proofs of the statements can be found in Section~\ref{statproof}.

\section{Variance}
In this section, we derive the variance of a random variable generated by an integer polygon. Furthermore, we will obtain a formula expressing the covariance of two random variables generated by integer polygons. 

Before formulating the main result about the variance, it is necessary to perform a series of transformations on the polygon. Assume that the integer points on the sides of the polygon are also its vertices and number them in counterclockwise order $a_1a_2\ldots a_N$. Let $v_i = \overrightarrow{a_ia_{i+1}}$. Suppose that $v_i = - v_j$ for some $i$ and $j$. Then the quadrilateral $a_ia_{i+1}a_{j-1}a_{j}$ is a parallelogram. By Statement~\ref{baza}, $X_{a_ia_{i+1}a_{j-1}a_j} = const$ a.s. Therefore, if we cut this parallelogram out of our polygon, the random variable will change to a constant, and so the variance will not change. Hence, without loss of generality, we can assume that the polygon has no parallel sides.
\begin{tm}\label{variance}
    Let $P$ be a polygon with no parallel sides and $l_1, l_2, \ldots l_n$ be their affine lengths. Then 
    \begin{align}
        \D X_P = \frac{l_1^2 + l_2^2 + \ldots + l_n^2}{12}. \nonumber
    \end{align}
\end{tm}
This theorem is a special case of the result for the covariance to which we turn now. First of all, let us introduce necessary notation. For two integer vectors $v, w \in \Z^2$ with affine lengths $|v|_{aff}, |w|_{aff}$ we define the function 
\begin{align}
    v \circ w = 
    \begin{cases}
        |v|_{aff}|w|_{aff}\quad &\text{ if } v \upuparrows w,\\
        -|v|_{aff}|w|_{aff}\quad &\text{ if } -v \upuparrows w,\\
        0\quad &\text{otherwise.}
    \end{cases}
\end{align}
\begin{tm}\label{covariance}
    Let $P$ and $Q$ be integer polygons with sides $v_1, v_2, \ldots v_{n_1}$  and $w_1, w_2, \ldots, w_{n_2}$ oriented in counterclockwise order. Then 
    \begin{align}
        \mathrm{cov} (X_P, X_Q) = \frac{1}{12} \sum_{i = 1}^{n_1} \sum_{j = 1}^{n_2} v_i \circ w_j.  \nonumber
    \end{align}
\end{tm}
The main key of the proof of this theorem is the following general fact. 
\begin{lm}\label{general}
    Let $C_1, C_2, \ldots C_n \subset \R^d$ be compact sets, and let $k_1, k_2, \ldots k_n \in \Z_{> 0}$. Then  
    \begin{align}
        \E \prod_{i=1}^n(X_{C_i} - \E X_{C_i})^{k_i} = \sum_{m_j^{i} \in \Z^d \setminus 0 \atop \sum m_j^{i} = 0} \prod_{i=1}^n \prod_{j = 1}^{k_i} \hat \ind_{C_i}(m_j^i), \nonumber
    \end{align}
    where $\ind_C$ is an indicator function and $\hat \ind_C$ is its Fourier transform. 
\end{lm}
In the special case for two sets and $k_1 = k_2 = 1$ we have
\begin{align}\label{twoind}
    \cov (X_{C_1}, X_{C_2}) = \E (X_{C_1} - \E X_{C_1})(X_{C_2} - \E X_{C_2}) = \sum_{m \in \Z^d\setminus 0} \hat \ind_{C_1}(m) \hat \ind_{C_2}(-m).
\end{align}

The connection between Equation~\ref{twoind} and Theorem~\ref{covariance}  is established by the following lemma.
\begin{lm}\label{indfour}
    Let $P$ be an integer polygon with sides $v_1, v_2, \ldots v_n$ oriented in the counterclockwise order. Then for any $m \in \Z^2 \setminus 0$,
    \begin{align}
        \hat \ind_P(m) = \frac{1}{2\pi i|m|_{aff}^2} \sum_{j = 1}^n m^{\perp} \circ v_j,
    \end{align}
    where $m^{\perp}$ is equal to $m$ rotated by $\pi/2$.
\end{lm}

Let us note the connection between \(\cov (X_{C_1}, X_{C_2})\) and the cross covariogram function. For two compact subsets \(A, B \in \mathbb{R}^d\), the \textit{cross covariogram function} is defined as the function \(g_{A, B}(x) = \mathrm{vol}(A \cap (B + x))\). In the case where the sets \(A\) and \(B\) coincide, the function is called the \textit{covariogram function} and is denoted by \(g_A(x)\). In \cite{Martins}, it was shown that for compact sets \(A\) and \(B\), the following equality holds:
\begin{align}
    \sum_{n \in \mathbb{Z}^d}g_{A, B}(n) = \vol(A) \vol(B) + \sum_{m \in \mathbb{Z}^d \setminus 0} \hat{\ind}_{A}(m) \bar{\hat{\ind}}_{B}(m).
\end{align}
Note that \(\bar{\hat{\ind}}_{B}(m) = \hat{\ind}_{B}(-m)\) and \(\mathrm{vol}(A) \mathrm{vol}(B) = \int_{\mathbb{R}^d}g_{A, B}(x) \, dx\). Therefore, using \eqref{twoind}, we obtain expressions for covariance and variance in terms of \(g_{A, B}\).
\begin{cor}
    Let \(A, B \subset \mathbb{R}^d\) be compact sets. Then:
    \begin{align}
        \cov(X_A, X_B) &= \sum_{n \in \mathbb{Z}^d}g_{A, B}(n) - \int_{\mathbb{R}^d}g_{A, B}(x) \, dx,  \nonumber\\
        \D X_A &= \sum_{n \in \mathbb{Z}^d}g_{A}(n) - \int_{\mathbb{R}^d}g_{A}(x) \, dx.   \nonumber
    \end{align}
\end{cor}
Using the observation \(\D X_A \geqslant 0\), we obtain the following corollary.
\begin{cor}
    Let \(A \subset \mathbb{R}^d\) be a compact set. Then
    \begin{align}
        \sum_{n \in \mathbb{Z}^d}g_{A}(n) \geqslant \int_{\mathbb{R}^d}g_{A}(x) \, dx = \mathrm{vol}(A)^2.  \nonumber
    \end{align}
\end{cor}

The full proofs of the lemmas and theorems presented in this paragraph can be found in Section~\ref{thproofs}.

\section{Integer triangle}
In this section, we describe the distribution of \(X_T\), when \(T\) is an integer triangle.

Let $T$ be an integer triangle with affine lengths $a, b$ and $c$. Assume that $\text{gcd} (a, b) = n$. Then $c$ is also divisible by $n$ and triangle $\frac 1 n T$ has an integer vertices. Note that $T = \frac 1 n T \oplus \frac 1 n T \oplus \ldots \oplus \frac 1 n T$, so by Statement \ref{baza} we have $X_T \ceq n X_{\frac 1 n T}$. Therefore, it is sufficient to consider triangles whose affine side lengths are coprime.

Before formulating the main result we introduce the notation. Let $n$ be an integer number. Denote by $U_n$ the random variable with uniform distribution on set $\{0, 1, \ldots, n-1\}$.

\begin{tm}
    Let $T$ be an integer triangle with affine lengths of sides $a, b, c$. Assume that $a, b$ and $c$ are coprime. Then 
    $$
        X_T \cdeq U_a + U_b + U_c + U_2,
    $$
    where all variables in the right--hand side are independent.
\end{tm}

\begin{proof}
    \textbf{Step 1.} First, we prove that $X_T$ takes at most $a+b+c-1$ different values. Indeed, let $a{v_1}, b{v_2}, c{v_3}$ be the sides of triangle (vectors are oriented in the counterclockwise order). Then by Statement~\ref{baza} and Statement~\ref{vecdist} we have
    \begin{align}\label{stepone}
        X_T \ceq aX_{v_1} + bX_{v_2} + cX_{v_3} \ceq a\lceil X \wedge v_1 \rceil + b\lceil X \wedge v_2 \rceil + c \lceil X \wedge v_3 \rceil.
  \end{align}
    Let \(\{\cdot\}\) denote the fractional part of a number. Then $\lceil X \wedge v \rceil = (X \wedge v ) + 1 - \{X \wedge v \}$ a.s. and we can continue (\ref{stepone}) as
    \begin{multline}
        a\lceil X \wedge v_1 \rceil + b\lceil X \wedge v_2 \rceil + c \lceil X \wedge v_3 \rceil =\nonumber\\
        = a+b+c + a(X \wedge v_1) + b(X \wedge v_2) + c(X \wedge v_3) \nonumber\\
        - a\{X \wedge v_1 \} - b\{X \wedge v_2 \} - c\{X \wedge v_3 \}.
    \end{multline}
    Using the linearity of $\wedge$ and the fact that $av_1 + bv_2 + cv_3 = 0$, we can simplify the formula: 
    \begin{align}
        a(X \wedge v_1) + b(X \wedge v_2) + c(X \wedge v_3) = X \wedge (av_1 + bv_2 + cv_3) = 0.
    \end{align} 
    Therefore,
    \begin{align}\label{steponeres}
        X_T \ceq a+b+c - a\{X \wedge v_1 \} - b\{X \wedge v_2 \} - c\{X \wedge v_3 \}. 
    \end{align}
    Moreover, the right--hand side of (\ref{steponeres}) is an integer number. It remains to note that this number is greater than 0 and less than $a+b+c$. Hence, the random variable in the right--hand side takes at most $a+b+c-1$ values, and $X_T$ has the same property.

    \textbf{Step 2.} Our next goal is to prove the following fact:
    \begin{align}\label{steptwo}
        X_T \text{ mod } a ~\deq~ U_a.
    \end{align}
    Similarly for other sides:
    \begin{align}
        X_T \text{ mod } b ~\deq~ U_b,    \nonumber
        \\
        X_T \text{ mod } c ~\deq~ U_c.    \nonumber
    \end{align}
    To prove (\ref{steptwo}), let us present the sides of $T$ as in step 1: $av_1$, $bv_2$, $cv_3$. Let $A \in SL(2, \Z)$ be an operator such that $Av_2 = e_1$ where $e_1 = (1, 0)^T$. Let us denote the images of $v_1$ and $v_3$ as $v_1'$ and $v_3'$: $Av_1 = v_1',$ $Av_3 = v_3'$.  Using Statement~\ref{baza} and (\ref{stepone}) we have
    \begin{align}
        X_T \deq X_{AT} \ceq a\lceil X \wedge v_1' \rceil + b \lceil X \wedge e_1 \rceil + c \lceil X \wedge v_3'  \rceil.
    \end{align}
    Consequently, 
    \begin{multline}
        X_T \text{ mod } a \cdeq ( a\lceil X \wedge v_1' \rceil + b \lceil X \wedge  e_1 \rceil + c \lceil X \wedge v_3' \rceil ) \text{ mod } a =\\= (b + c \lceil X \wedge v_3' \rceil) \text{ mod } a.
    \end{multline}
    Since $a, b, c$ are coprime, it remains to prove that $\lceil X \wedge v_3' \rceil \text{ mod } a \deq U_a$. To do this, we note that $av_1' + be_1 + cv_3' = 0$. The second coordinate of $av_1 + be_1$ is divisible by $a$, so the second coordinate of $v_3'$ is also divisible by $a$. Hence, $X \wedge v_3' $ can be rewritten as $alX_1 - kX_2$, where $v_3' = (k, al)$ and $X = (X_1, X_2)^T$. Note that for any fixed $X_2 = x_2$ the random variable $\lceil alX_1-kx_2 \rceil \text{ mod } a$ has uniform distribution on $\{0, 1, \ldots a-1\}$. Therefore, the random variable  $\lceil  alX_1 - kX_2\rceil \text{ mod } a$ also has the same distribution. Hence, $\lceil X \wedge v_3' \rceil \text{ mod } a \deq U_a$ and we have completed Step 2.

    \textbf{Step 3.} In this step we will finish the proof of the theorem by using the generating function technique. According to the first step, $X_T$ takes at most $a+b+c$ consecutive integer values. Consequently, there is an integer constant $C$ such that $X_T + C \in \{0, 1, \ldots a+b+c-1\}$ a.s. Moreover, we can assume that $\Prob \{X_T + C = 0\} > 0$. We define $X_T' = X_T + C$ and note that $X_T \ceq X_T'$. By Step 2, $X_T \text{ mod }a \deq U_a$, so $X_T'$ has the same property. Consider a function
    \begin{align}
        G(z) = \E z^{X_T'}.
    \end{align}
    Note that $G(z)$ is a polynomial of degree at most $a+b+c-2$ by Step 1. Moreover, $G(\xi) = 0$ for any $\xi$ such that $\xi ^ a = 1$ by Step 2. Hence $G(z)$ is divisible by  $F_a(z) = \frac 1 a (1 + z + z^2 + \ldots z^{a-1})$. Since $a, b, c$ are coprime, $G$ can be represented as 
    \begin{align}
        G(z) = F_a(z)F_b(z)F_c(z)H(z),
    \end{align}
    where $H(z)$ is a polynomial of degree at most 1. By the assumption, $\Prob\{X_T' = 0\} \ne 0$, so $G(0) \ne 0$. Moreover, $X_T'$ has a symmetric distribution by Statement~\ref{baza}, hence the leading coefficient of $G$ is equal to $G(0)$. Therefore $H(z)$ can be either a constant or a polynomial of the form 
    $$
        H(z) = \frac 1 2 (1 + z) = F_2(z).
    $$
    In the first case we have $G(z) = F_a(z)F_b(z)F_c(z)$ and $X_T' \deq U_a + U_b + U_c$. Then $\D X_T = \D X_T' = \frac{a^2 + b^2 +c^2-3}{12}$, but this contradicts Theorem~\ref{variance}. Consequently, $G(z) = F_a(z)F_b(z)F_c(z)F_2(z)$ and $X_T \ceq X_T' \deq U_a + U_b + U_c + U_2$.
    
\end{proof}

\textbf{Remark 1.} The end of the proof can be done without Theorem~\ref{variance}. Indeed, let $G(z) = F_a(z)F_b(z)F_c(z)$. Then $X_T' \deq U_a + U_b + U_c$ and $\E X_T' = \frac{a+b+c - 3}{2}$. Since $X_T$ takes integer values, there exists an integer constant $k$ such that $X_T = X_T' + k$ a.s. and $\E X_T = \frac{a+b+c - 3}{2} + k$. Moreover, $\E X_T = \area(T)$ by Statement~\ref{baza}, so 
\begin{align}\label{rem11}
    \area(T) = \frac{a+b+c - 3}{2} + k.
\end{align}
On the other hand, by Pick's theorem,
\begin{align}\label{rem12}
    \area(T) = \frac{a + b + c - 2}{2} + i,
\end{align}
where $i$ is the number of integer points inside $T$. It follows from (\ref{rem11}) and (\ref{rem12}) that $\frac 1 2$ must be an integer, a contradiction. Therefore, $X_T' \ne U_a + U_b + U_c$.

\textbf{Remark 2.} At the first step of the proof, it was demonstrated that the number of potential values that $X_T$ can take is bounded from above by $a + b + c - 1$. A similar conclusion can be reached for a polygon with an arbitrary number of sides by employing a parallel line of reasoning.
\begin{cor}
    Let $\per_{aff}(P)$ be the affine perimeter of an integer polygon $P$. Then $X_P$ can take at most $\per_{aff}(P) - 1$ different values.
\end{cor}

\section{Proofs of Theorem~\ref{variance} and Theorem~\ref{covariance}}\label{thproofs}
This section is organized as follows. First, we prove Lemma~\ref{general} and Lemma~\ref{indfour}. Then we derive Theorem~\ref{covariance} from this lemmas and have Theorem~\ref{variance} as a corollary.

\begin{proof}[Proof of Lemma~\ref{general}]
Let us first concentrate on the random variable $X_C$ for compact set $C \subset \R^d$. By definition of $X_C$, we have
\begin{align}
    X_C(\bx) = \sum_{m \in \Z^d}\ind_{C + \bx}(m) = \sum_{m \in \Z^d}\ind_{C}(m - \bx). \nonumber
\end{align}
By the Poisson summation formula, 
\begin{align}
    \sum_{m \in \Z^d}\ind_{C}(m - \bx) = \sum_{m \in \Z^d}\hat\ind_{C}(m)e^{2\pi i \langle m, - \bx \rangle}.\nonumber
\end{align}
Hence, we have
\begin{align}\label{thproofone}
    X_C(\bx) = \sum_{m \in \Z^d}\hat\ind_{C}(m)e^{2\pi i \langle m, - \bx \rangle}.
\end{align}
Note, that 
\begin{align}\label{thprooftwo}
    \int_{[0,1]^d}e^{2\pi i \langle m, - \bx \rangle} d\bx = 
    \begin{cases}
        0 \quad &\text{if }m \ne 0,\\
        1 \quad &\text{if }m = 0.
    \end{cases}
\end{align}
Therefore, 
\begin{align}\label{thproofthree}
    \E X_C = \int_{[0,1]^d} X_C(\bx)~d\bx = \int_{[0,1]^d}\sum_{m \in \Z^d}\hat\ind_{C}(m)e^{2\pi i \langle m, - \bx \rangle}~d\bx = \hat \ind_C(0).
\end{align} 
Combining (\ref{thproofone}) and (\ref{thproofthree}) we have
\begin{align}
    X_C(\bx) - \E X_C = \sum_{m \in \Z^d\setminus 0}\hat\ind_{C}(m)e^{2\pi i \langle m, - \bx \rangle}.
\end{align}
Now we can rewrite the left-hand side of the statement of the lemma as
\begin{multline}
    \E \prod_{j = 1}^n (X_{C_j} - \E X_{C_j})^{k_j} = \int_{[0,1]}\prod_{j = 1}^n \left(\sum_{m \in \Z^d\setminus 0}\hat\ind_{C_j}(m)e^{2\pi i \langle m, - \bx \rangle}\right)^{k_j} ~d\bx =\\=
    \int_{[0,1]^d} \sum_{m_l^{j} \in \Z^d \setminus 0 } \left( \prod_{j=1}^n \prod_{j = l}^{k_j} \hat \ind_{C_i}(m_j^i) \right)e^{\langle \sum m^j_l, -\bx\rangle}~d\bx.
\end{multline}
All summands with non-zero $m^j_l$ vanish according to (\ref{thprooftwo}). Therefore,
\begin{align}
    \int_{[0,1]^d} \sum_{m_l^{j} \in \Z^d \setminus 0 } \left( \prod_{j=1}^n \prod_{j = l}^{k_j} \hat \ind_{C_i}(m_j^i) \right)e^{\langle \sum m^j_l, -\bx\rangle}~d\bx =  \sum_{m_l^{j} \in \Z^d \setminus 0 \atop \sum m^j_l = 0} \left( \prod_{j=1}^n \prod_{j = l}^{k_j} \hat \ind_{C_i}(m_j^i) \right),
\end{align}
and we have now completed the proof of Lemma~\ref{general}.

\end{proof}

\begin{proof}[Proof of Lemma~\ref{indfour}]
Let us consider two functions:
\begin{align}
    &f(P, m) = \hat \ind_P(m);  \nonumber\\ 
    &g(P, m) = \frac{1}{2\pi i |m|^2_{aff}} \sum_{j = 1}^n m^{\perp} \circ v_j. \nonumber
\end{align}
Our goal is to prove that $f(P, m) = g(P, m)$ for any integer polygon $P$ and $m \in \Z^2\setminus 0$. First, note that if $P = P_1 \cup P_2$ for integer polygons such that $\mathrm{int} P_1 \cap \mathrm{int} P_2 = \varnothing$ then for any $m \in \Z \setminus 0$ we have
\begin{align}
    f(P, m) = f(P_1, m) + f(P_2, m);\label{funf}\\
    g(P, m) = g(P_1, m) + g(P_2, m).\label{fung}
\end{align}
Indeed, for the function $f$ we have 
\begin{multline}
    f(P, m) = \hat \ind_P(m) = \int_{P} e^{-2\pi i\langle m, t \rangle}~dt =\\= \int_{P_1} e^{-2\pi i\langle m, t \rangle}~dt + \int_{P_2} e^{-2\pi i\langle m, t \rangle}~dt = f(P_1, m) + f(P_2, m).
\end{multline}
To prove (\ref{fung}), we denote the common sides of polygons $P_1$ and $P_2$ as $w_1, w_2, \ldots w_k$ for $P_1$ and $-w_1, -w_2, \ldots -w_k$ for $P_2$. By definition of $\circ$ we have $m^\perp \circ w = -m^\perp \circ (-w) $, therefore
\begin{multline}
    g(P, m) = \frac{1}{2\pi i |m|^2_{aff}} \sum_{j = 1}^n m^{\perp} \circ v_j = \\=\frac{1}{2\pi i |m|^2_{aff}} \left(\sum_{j = 1}^n m^{\perp} \circ v_j + \sum_{l=1}^k m^{\perp} \circ w_l + \sum_{l=1}^k m^{\perp} \circ (-w_l)\right) =\\= g(P_1, m) + g(P_2,m).\nonumber
\end{multline}
According to (\ref{funf}) and (\ref{fung}), we can consider that $P$ is a triangle with no integer points on its sides and inside. Furthermore, the problem can be reduced to a fixed triangle by making the following observation. Let $A \in SL(2, \Z)$, then 
\begin{align}
    f(AP, m) = f(P, A^Tm),\label{funfA}\\
    g(AP, m) = g(P, A^Tm).\label{fungA}
\end{align}
Indeed, to prove (\ref{funfA}), it is enough to change the variables in the integral:
\begin{multline}
    f(AP, m) = \hat \ind_{AP}(m) = \int_{AP} e^{-2\pi i\langle m, t \rangle}~dt =\\= \int_P e^{-2\pi i\langle m, As \rangle}~ds = \int_P e^{-2\pi i\langle A^Tm, s \rangle}~ds = f(P, A^Tm).\nonumber
\end{multline}
To prove (\ref{fungA}) we need to check that $|Am|_{aff} = |m|_{aff}$ and $m^\perp \circ (Av_j) = (A^Tm)^\perp \circ v_j$. Since $A\Z^2 = \Z^2$, the number of points on $m$ is equal to the number of points on $Am$. Therefore $|Am|_{aff}=|m|_{aff}$. Similarly, $|(A^Tm)^\perp|_{aff} = |m^\perp|$ and $|Av_j|_{aff} = |v_j|_{aff}$. Moreover,
\begin{align}
    m^\perp \circ (Av_j) \ne 0 \iff \langle m, Av_j \rangle = 0 \iff \langle A^Tm, v_j\rangle = 0 \iff (A^Tm)^\perp \ne 0.
\end{align}
Hence, $m^\perp \circ (Av_j) = (A^Tm)^\perp \circ v_j$ and $g(AP, m) = g(P, Am)$.

Now we are ready to complete the proof of Lemma~\ref{indfour}. As previously indicated, it can be assumed that $P$ is an integer triangle with no integer points on its sides and interior. By Pick's theorem, the area of such triangle is equal to $\frac 1 2$. Hence, there exists an operator $A \in SL(2, \Z)$ such that $A\Delta = P$, where $\Delta$ is a triangle with vertex coordinates equal to $(0,0); (0,1);(1,0)$. Moreover, $f(P, m) = f(A\Delta, m) = f(\Delta, A^Tm)$ and $g(P, m) = g(A\Delta, m) = g(\Delta, A^Tm)$. Therefore, it is sufficient to verify that $f(\Delta, m) = g(\Delta, m)$ for any $m \in \mathbb{Z}^2 \setminus 0$. By simple calculations,
\begin{align}
    f(\Delta, m) = 
    \begin{cases}
        \frac{1}{2\pi i k}  \quad &\text{if } m = (0,k)^T, (k, 0)^T\text{ or } (k, k)^T,\\
        0   \quad &\text{otherwise}.
    \end{cases}
\end{align}
By definition, $g(\Delta, m)$ takes the same values, which completes the proof.

\end{proof}
%______________________________________
\begin{proof}[Proof of Theorem~\ref{covariance}]
According to Lemma~\ref{general} we have
\begin{align}\label{thprooffour}
    \cov (X_P, X_Q) = \E (X_P - \E X_P)(X_Q - \E X_Q) = \sum_{m \in \Z^2\setminus 0} \hat \ind_P(m) \hat \ind_Q(-m).
\end{align}
Using Lemma 2, we can continue the equality
\begin{align}\label{thprooffive}
    \sum_{m \in \Z^2\setminus 0} \hat \ind_P(m) \hat \ind_Q(-m) = \sum_{m \in \Z^2 \setminus 0} -\frac 1 {4\pi^2|m|^4_{aff}}\left(\sum_{i = 1}^{n_1} \sum_{j = 1}^{n_2}(m^{\perp}\circ v_i)(-m^{\perp} \circ w_j)\right).
\end{align}
Note that if $v \nparallel w$, then $(m^{\perp}\circ v)(-m^{\perp} \circ w) = 0$. Alternatively, if $v \parallel w$, then
\begin{align}
    (m^{\perp}\circ v)(-m^{\perp} \circ w) = 
    \begin{cases}
        -|m|_{aff}^2|v|_{aff}|w|_{aff} \quad &\text{if } m \parallel v,\\
        0 \quad &\text{otherwise.}\nonumber
    \end{cases}
\end{align}
It is therefore possible to transform the following sum: 
\begin{multline}\label{thproofsix}
    \sum_{m \in \Z^2\setminus 0} -\frac 1 {4\pi^2|m|^4_{aff}} (m^{\perp}\circ v)(-m^{\perp} \circ w) = 
    \sum_{m \in \Z^2\setminus 0 \atop m \parallel v} \frac {|m|_{aff}^2 v \circ w} {4\pi^2|m|^4_{aff}} = \\=
    v \circ w \cdot \sum_{k \in \Z \setminus 0}\frac 1{4\pi^2 k^2} = \frac {v \circ w} {12}.
\end{multline}
Combining (\ref{thprooffour}), (\ref{thprooffive}) and (\ref{thproofsix}) we have
\begin{multline}
    \cov (X_P, X_Q) =\\
    =\sum_{m \in \Z^2 \setminus 0} -\frac 1 {4\pi^2|m|^4_{aff}}\left(\sum_{i = 1}^{n_1} \sum_{j = 1}^{n_2}(m^{\perp}\circ v_i)(-m^{\perp} \circ w_j)\right)=\\ = \frac 1 {12} \sum_{i = 1}^{n_1} \sum_{j = 1}^{n_2} v_i\circ w_j, \nonumber
\end{multline}
which completes the proof.

\end{proof}

\begin{proof}[Proof of Theorem~\ref{variance}]
    By assumption, the polygon has no parallel sides, so $v_i \circ v_j = 0$ for $i \ne j$. Then by Theorem~\ref{covariance} we have
    \begin{align}
        \D X_P = \cov (X_P, X_P) = \frac{1}{12} \sum_{i = 1}^n \sum_{i = 1}^n v_i \circ v_j = \frac 1 {12} \sum_{i = 1}^n v_i \circ v_i = \frac 1 {12} \sum_{i = 1}^n l_i^2. \nonumber
    \end{align}
\end{proof}

\section{Proofs of Statement~\ref{vecdist} and Statement~\ref{baza}}\label{statproof}
\begin{proof}[Proof of Statement~\ref{vecdist}]
    Let us first note that $X_v = |v|_{aff}X_{v/|v|_{aff}}$ by definition. Therefore, we can assume that $|v|_{aff} = 1$ and prove that $X_v = \lceil X \wedge v\rceil$.

    Note that $m \wedge v \in \Z$ for any integer point $m \in \Z^2$. Let us fix $k \in \Z^2$ and сonsider the set $L_k = \{m \in \Z^2 \colon m \wedge v = k\}.$ All points in $L_k$ lie on the line parallel to $v$. Furthermore, the distance of any two points in $L_k$ is at least $|v|_{\R^2}$. Therefore, any parallelogram $x \oplus v$ cane take at most 1 point from $L_k$. If $k < 0$ then $x\oplus v \cap L_k \ne \varnothing$ if and only if $x\wedge v \leqslant k$. Similarly for $k > 0$, $x \oplus L_k \ne \varnothing$ if and only if $x \wedge v \geqslant k$. Therefore,
    \begin{align}
        X_v = |\{k \in \Z \colon (X\oplus v) \cap L_k \ne \varnothing| + \ind_{\geqslant 0}(X \wedge v) = \lceil X \wedge v \rceil. \nonumber
    \end{align}
\end{proof}
\begin{proof}[Proof of Statement~\ref{baza}]
1. By definition,
\begin{align}
    X_P = \sum_{m \in \Z^2} \ind_{P+X}(m) = \sum_{m \in \Z^2} \ind_P(m - X).
\end{align}
Note that 
\begin{align}
    \E \ind_P(m - X) = \int_{m - [0,1]^2}\ind_P~d\bx = \area\left((m-[0,1]^2) \cap P\right).
\end{align}
Therefore, 
\begin{multline}
    \E X_P = \sum_{m \in \Z^2} \area\left((m-[0,1]^2) \cap P\right) = \area \left(\left(\bigcup_{m \in \Z^2}(m-[0,1]^2)\right) \cap P\right) =\\= \area(\R^2 \cap P) = \area(P).
\end{multline}

2. Note that $A\Z^2 = Z^2$ and $X \deq AX/\Z^2$. Therefore,
\begin{multline}
    X_{AP} = \sum_{m \in \Z^2} \ind_{AP+X}(m) =
    \sum_{m \in \Z^2} \ind_{AP+X}(Am) \deq \\ \deq
    \sum_{m \in \Z^2} \ind_{AP+AX}(Am) = 
    \sum_{m \in \Z^2} \ind_{A+X}(m) = 
    X_P.
\end{multline}

3. Similar to item 2.

4. Let $v_1, v_2, \ldots v_n$ be the sides of $P$ and $a_1, a_2, \dots a_n$ be its vertices. Then, 
$$
\ind_{P+\bx} - \ind_{P} \equiv \sum_{i = 1}^n \text{sgn} (v_i \wedge \bx)\ind_{a_i + v_i \oplus \bx}.
$$
If we integrate this equation over $\Z^2$ with respect to the counting measure, we obtain
$$
    |(P+\bx) \cap \Z^2| - |P \cap \Z^2| = \sum_{i = 1}^n \text{sgn} (v_i \wedge \bx) |v_i\oplus \bx \cap P|.
$$
Therefore, 
$$
    X_P - |P \cap \Z^2| = \sum_{i = 1}^n X_{v_i},
$$
and hence,
$$
    X_P \ceq \sum_{i = 1}^n X_{v_i}.
$$

5. If $\{v_i\}_{i = 1}^{n_1}$ are the sides of $P$ and $\{w_j\}_{i = 1}^{n_2}$ are the sides of $Q$ then $\{v_i\}_{i = 1}^{n_1} \cup \{w_j\}_{i = 1}^{n_2}$ are the sides of $P \oplus Q$. Therefore, according to item 4, we have
$$
    X_{P\oplus Q} \ceq \sum_{i = 1}^{n_1} X_{v_i} + \sum_{j = 1}^{n_2}X_{w_j} \ceq X_P + X_Q.
$$

6. The centrally symmetric polygon is a sum of integer segments: $P = s_1\oplus s_2\oplus\ldots\oplus s_n$. Any segment $s$ is a polygon with two sides and $X_s = 0$ almost surely. Then $X_P \ceq X_{s_1} + X_{s_2} + \ldots + X_{s_n} = 0$ a.s.

7. By item 3, we have $X_P \deq X_{-P}$. Note that $P \oplus(-P)$ is a centrally symmetric polygon. So, according to items 5 and 6,
$$
X_P + X_{-P} \ceq X_{P\oplus(-P)} \ceq 0.
$$
Therefore,
$X_P \ceq -X_{-P} \deq -X_P,$
so $X_P \cdeq -X_P$.

\end{proof}


\begin{thebibliography}{1}
    \bibitem{Strombergsson} Andreas Str\"{o}mbergsson and Anders S\"{o}dergren, \emph{On the generalized circle problem for a random lattice in large dimension}, Advances in Mathematics \textbf{345} (2016). 

    \bibitem{Janacek}Jir\'{\i} Jan\'{a}\v{c}ek, \emph{Asymptotics of variance of the lattice point count}, Czech Math J \textbf{58} (2008): 751–758 . 

    \bibitem{Pick} Georg Pick, \emph{Geometrisches zur Zahlenlehre}, Sitzungsberichte des deutschen naturwissenschaftlich-medicinischen Vereines für Böhmen "Lotos" in Prag. (Neue Folge) \textbf{19} (1899): 311–319.

    \bibitem{Kendall} David Kendall, \emph{On the number of lattice points inside a random oval}, Quarterly Journal of Mathematics \textbf{1} (1948): 1-26.

    \bibitem{Martins}Michel Faleiros Martins and Sinai Robins, \emph{The covariogram and extensions of the Bombieri-Siegel formula} arXiv preprint arXiv:2204.08606 (2022).

    \bibitem{Santalo}Llu{\'\i}s Santal{\'o}, \emph{Integral geometry and geometric probability} (1976).
\end{thebibliography}
\end{document}